\documentclass{article}
\usepackage{graphicx}
\usepackage{amsmath,amssymb, amsthm}
\usepackage{algorithm}
\usepackage{algorithmic}
\usepackage{tikz}
\usepackage[margin=10truemm]{geometry}
\newtheorem{theorem}{Theorem}
\newtheorem{lemma}{Lemma}
\newtheorem{definition}{Definition}
\newtheorem{corollary}{Corollary}

\title{Improving upper and lower bounds of the number of games born by day 4}

\author{Koki Suetsugu \\ National Institute of Informatics}
\date{}
\begin{document}
\maketitle

\begin{abstract} In combinatorial game theory, the lower and upper bounds of the number of games born by day $4$ have been recognized as $3.0 \cdot 10^{12}$ and $10^{434}$, respectively. In this study, we improve the lower bound to $10^{28.2}$ and the upper bound to $4.0 \cdot 10^{184}$, respectively.
\end{abstract}

\section{Introduction}

The main result of this study is improving the upper and lower bounds on the total number of canonical forms born by day $4$ under normal play convention.
Games with game tree heights less than or equal to $n$ are called games born by day $n$, and the set of all canonical forms born by day $n$ is denoted by $\mathbb{G}_n$. So far, the total number of canonical forms born by day $0, 1, 2, 3$ have been recognized as $1, 4, 22, 1474$, respectively. Meanwhile, the total number of canonical forms born by day $4$ is vast, and only wide upper and lower bounds $10^{434}$ and $3\cdot 10^{12}$ are known. Improving these upper and lower bounds was discussed as an open problem in \cite{CGT} and \cite{Now19} and has created much attention.
In this study, we used algebraic properties and programming to improve the upper and lower bounds and obtained a new upper bound, $4.0 \cdot 10^{184}$ and a new lower bound, $10^{28.2}$.

\subsection{Early results}
\label{sec2}

Early results of counting the number of games are summarized as a note in \cite{CGT}, Chapter 3, Section 1. According to the note, the total number of canonical forms born by day $3$ and the total numbers of dicotic games born by day $4$, reduced games born by day $4$, and hereditarily transitive games born by day $4$ are known, whereas the total number of canonical forms born by day $4$ is unknown.
As noted in the note, upper and lower bounds of general $|\mathbb{G}_n|$ were obtained by \cite{FW04}. 
The upper bound is given in the following three forms, with the lower forms having more complicated inequalities but stricter upper bounds.
$$
|\mathbb{G}_{n+1}| \leq 2^{|\mathbb{G}_n| + 1}  + |\mathbb{G}_n|.
$$
$$
|\mathbb{G}_{n+1}| \leq |\mathbb{G}_n| + 2^{|\mathbb{G}_n|} + 2.
$$
$$
|\mathbb{G}_{n+1}| \leq |\mathbb{G}_n| + [|\mathbb{G}_{n-1}|^2 + \frac{5}{2}|\mathbb{G}_{n-1}| + 2]2^{|\mathbb{G}_n|-2|\mathbb{G}_{n-1}|}.
$$

Substituting $|\mathbb{G}_3| = 1474, |\mathbb{G}_2| = 22$ into the third equation and examining the values specifically, we obtain

\begin{eqnarray}
|\mathbb{G}_4| &\leq& 1474 + [22^2 + \frac{5 \cdot 22}{2} + 2]2^{1474 - 2\cdot 22}\nonumber \\
&=& 1474 + [484 + 55 + 2] 2^{1430} \nonumber \\
&<& 10^{3} \cdot 10^{0.3011 \cdot 1430} \nonumber \\ 
&<& 10^{3} \cdot 10^{430.573} \nonumber \\
&=& 10^{433.573} \nonumber \\
&<&10^{434}. \nonumber
\end{eqnarray}

The lower bound is given in the following two forms, with the lower form having complicated inequality but a stricter lower bound.

$$
|\mathbb{G}_{n+1}|\geq 2^{\frac{|\mathbb{G}_n|}{2|\mathbb{G}_{n-1}|}}.
$$
$$
|\mathbb{G}_{n+1}|\geq 	(8|\mathbb{G}_{n-1}|-4)( 2^{\frac{|\mathbb{G}_n| - 2}{2|\mathbb{G}_{n-1}|-1}} -1).
$$

Substituting $|\mathbb{G}_3| = 1474, |\mathbb{G}_2| = 22$ into the second equation, we obtain,

\begin{eqnarray}
|\mathbb{G}_4|&\geq& (8 \cdot 22 - 4) \cdot (2 ^ \frac{1474 - 2}{2 \cdot 22 - 1} -1) \nonumber \\
&=& 172 \cdot (2^{\frac{1472}{43}} - 1) \nonumber \\
&>& 171 \cdot 10^{\frac{1472\cdot 0.301}{43}} \nonumber \\
&=& 171 \cdot 10^{10.304} \nonumber \\
&>& 171 \cdot 2 \cdot 10^{10} \nonumber \\
&>& 3.0 \cdot 10^{12} \nonumber.
\end{eqnarray}

Thus, we have
$$
3.0 \cdot 10^{12} < |\mathbb{G}_4| < 10^{434}, 
$$
and the width of upper and lower bounds is vast.

The rest of the paper is organized as follows. Section \ref{sec3} studies properties of $\mathbb{G}_3$ for improving upper and lower bounds of $|\mathbb{G}_4|$.
Using these properties, Section \ref{sec4} presents improved upper and lower bounds.
To guess which of the upper and lower bounds obtained is closer to the true value, in Section \ref{sec5}, we calculate the upper and lower bounds of $\mathbb{G}_3$ using the same method and compare them with the true value.
Section \ref{sec6} provides a conclusion.
In the Appendix, all the elements of $\mathbb{G}_3$ used in this improvement are summarized in tables.

\section{Division of $\mathbb{G}_3$}
\label{sec3}
\subsection{Stratification of $\mathbb{G}_3$}
In this section, to use improving for upper and lower bounds of $|\mathbb{G}_4|$, we divide $\mathbb{G}_3$ by two methods which have good properties.
 First, by using CGSuite(Version 0.7), we found all the elements of $\mathbb{G}_3$ and divided them as Algorithm \ref{algo1}.
\begin{algorithm}[tb]
\caption{Algorithm determining stratification of $\mathbb{G}_3$}
\label{algo1}
\begin{algorithmic}
\STATE $S \leftarrow \mathbb{G}_3$
\STATE $i \leftarrow 0$
\WHILE{$|S| > 0$}
\STATE $i \leftarrow i + 1$
\STATE $U_i \leftarrow \emptyset $
\FORALL{$s \in S$}
 \IF {$\forall t \in S \  (s \not < t)$}
\STATE $U_i = U_i \cup \{s\}$
\ENDIF
\ENDFOR
\STATE $S \leftarrow S \setminus U_i$
\ENDWHILE
\end{algorithmic}
\end{algorithm}


Intuitively, the elements of $\mathbb{G}_3$ are divided into groups, in order from largest to smallest.
Each pair of elements in the same group is incomparable and each element except for the element in the first group has at least one larger element in the upper group of the group it belongs to. 
From this calculation, we obtained the sets $U_1, U_2, \cdots, U_{45}$. We call them a {\em stratification} of $\mathbb{G}_3$.
The number of elements in each set is shown in the Table \ref{tableU}.
For reference, the results for $\mathbb{G}_2$ of a similar process are shown in Fig \ref{u_2}. For $\mathbb{G}_3$, the number of elements is enormous, so a schematic is shown in Figure \ref{u_3}.
 Elements connected by lines have an order, with the upper element being larger than the lower element. Elements not directly or indirectly connected to each other are not comparable.

\begin{table}[tb]
\caption{Number of elements in $U_i$}
\label{tableU} 
\begin{tabular}{r|cccccccccccccccccccccc}
\hline
\hline
$i$ & 1 & 2 & 3 & 4 & 5 & 6 & 7 & 8 & 9 & 10 & 11 & 12 & 13 & 14 & 15 & 16 & 17 & 18 & 19 & 20 & 21 & 22\\ \hline
$|U_i|$ & 1 & 2 & 3 & 5 & 8 & 9 & 12 & 14 & 17 & 20 & 24 & 26 & 30 & 34 & 39 & 45 & 52 & 58 & 65 & 72  & 77 & 81\\
\end{tabular}


\begin{tabular}{ccccccccccccccccccccccc}
\hline
\hline
23 & 24 & 25 & 26 & 27 & 28 & 29 & 30 & 31 & 32 & 33 & 34 & 35 & 36 & 37 & 38 & 39 & 40 & 41 & 42 & 43 & 44 & 45 \\ \hline
86 & 81 & 77 & 72 & 65 & 58 & 52 & 45 & 39 & 34  & 30 & 26 &  24 & 20 & 17 & 14 & 12 & 9 & 8 & 5 & 3 & 2  & 1\\
\end{tabular}


\end{table}

\begin{figure}[tb]
\begin{tikzpicture}[auto]
\node (a) at (1.5, 0) {$-2$}; 
\node (b) at (0, 1.5) {$-1$};
\draw[-] (a) to (b);
\node (c) at (3, 1.5) {$-1*$};
\draw[-] (a) to (c);
\node (d) at (1.5, 3) {$-\frac{1}{2}$};
\node (e) at (3, 3) {$\{*\mid -1\}$};
\node (f) at (4.5, 3) {$\{0 \mid -1\}$};
\draw[-] (b) to (d);
\draw[-] (c) to (d);
\draw[-] (c) to (e);
\draw[-] (c) to (f);
\node (g) at (1.5, 4.5) {$\downarrow$};
\node (h) at (3, 4.5) {$\downarrow *$};
\node (i) at (4.5, 4.5) {$\{0, * \mid -1\}$};
\draw[-] (d) to (g);
\draw[-] (d) to (h);
\draw[-] (e) to (g);
\draw[-] (e) to (i);
\draw[-] (f) to (h);
\draw[-] (f) to (i);
\node (j) at (0, 6) {$0$};
\node (k) at (1.5, 6) {$*$};
\node (l) at (3,6) {$*2$};
\node (m) at (4.5,6) {$\pm 1$};
\draw[-] (g) to (j);
\draw[-] (g) to (l);
\draw[-] (h) to (k);
\draw[-] (h) to (l);
\draw[-] (i) to (l);
\draw[-] (i) to (m);
\node (n) at (1.5, 7.5) {$\uparrow$};
\node (o) at (3, 7.5) {$\uparrow*$};
\node (p) at (4.5, 7.5) {$\{1 \mid 0,*\}$};
\draw[-] (j) to (n);
\draw[-] (k) to (o);
\draw[-] (l) to (n);
\draw[-] (l) to (o);
\draw[-] (l) to (p);
\draw[-] (m) to (p);
\node (q) at(1.5, 9) {$\frac{1}{2}$};
\node (r) at(3,9) {$\{1 \mid * \}$};
\node (s) at (4.5,9) {$\{1 \mid 0\}$};
\draw[-] (n) to (q);
\draw[-] (n) to (r);
\draw[-] (o) to (q);
\draw[-] (o) to (s);
\draw[-] (p) to (r);
\draw[-] (p) to (s);
\node (t) at(0, 10.5) {$1$};
\node(u) at (3,10.5){$1*$};
\draw[-] (q) to (t);
\draw[-] (q) to (u);
\draw[-] (r) to (u);
\draw[-] (s) to (u);
\node (v) at(1.5,12) {$2$};
\draw[-] (t) to (v);
\draw[-] (u) to (v);
\node (u9) at (7, 0 ) {$U_9 = \{-2\}$};
\node (u8) at (7,1.5){$U_8 = \{-1, -1*\}$};
\node (u7) at (7, 3.0){$U_7 = \{ -\frac {1}{2}, \{*\-1\}, \{0|-1\}\}$};
\node (u6) at (7, 4.5){$U_6 = \{ \downarrow, \downarrow *, \{0,*|-1\}\}$};
\node (u5) at (7, 6) {$U_5 = \{ 0, *, *2, \pm1\}$};
\node (u4) at (7, 7.5) {$U_4 = \{ \uparrow, \uparrow*, \{1|0,*\}\}$};
\node (u3) at (7, 9) {$U_3 = \{\frac{1}{2}, \{1|*\}, \{1|0\}\}$};
\node (u2) at (7, 10.5) {$U_2 = \{1,1*\}$};
\node (u1) at (7,12) {$U_1=\{2\}$}; 
\end{tikzpicture}
\caption{$22$ games born by day $2$ and their stratification}
\label{u_2}
\end{figure}

\begin{figure*}[tb]
\begin{center}
\begin{tikzpicture}[auto]
\node (a) at (1.5, 0) {$-3$}; 
\node (b) at (0, 1.5) {$-2$};
\draw[-] (a) to (b);
\node (c) at (3, 1.5) {$-2*$};
\draw[-] (a) to (c);

\node (d) at (1.5, 3) {$-\frac{3}{2}$};
\node (e) at (3, 3) {$\{-1|-2\}$};
\node (f) at (4.5, 3) {$\{-1* | -2\}$};
\draw[-] (b) to (d);
\draw[-] (c) to (d);
\draw[-] (c) to (e);
\draw[-] (c) to (f);

\node (j) at (0, 4.5) {$\vdots$};
\node (k) at (1.5,4.5) {$\vdots$};
\node (l) at (3,4.5) {$\vdots$};
\node (m) at (4.5,4.5) {$\vdots$};
\node (n) at (6,4.5) {$\vdots$};

\draw[-] (d) to (j);
\draw[-] (d) to (k);
\draw[-] (e) to (k);
\draw[-] (e) to (l);
\draw[-] (f) to  (j);
\draw[-] (f) to (l);
\draw[-] (f) to (m);
\draw[-] (f) to (n);

\node (g) at (1.5, 6) {$\frac{3}{2}$};
\node (h) at (3, 6){$\{2|1\}$};
\node (i) at (4.5,6) {$\{2|1*\}$};

\draw[-] (j) to (g);
\draw[-] (k) to (g);
\draw[-] (k) to (h);
\draw[-] (l) to (h);
\draw[-] (j) to (i);
\draw[-] (l) to (i);
\draw[-] (m) to (i);
\draw[-] (n) to (i);

\node (t) at(0, 7.5) {$2$};
\node(u) at (3,7.5){$2*$};
\draw[-] (g) to (t);
\draw[-] (g) to (u);
\draw[-] (h) to (u);
\draw[-] (i) to (u);
\node (v) at(1.5,9) {$3$};
\draw[-] (t) to (v);
\draw[-] (u) to (v);
\node(u1) at (9,9) {$U_1 = \{3\}$};
\node(u2) at (9,7.5) {$U_2 = \{2, 2*\}$};
\node(u3) at (9,6) {$U_3 =\{ \frac{3}{2}, \{2|1\}, \{2|1*\}\} $};
\node(udots) at (9,4.5) {$\vdots$};
\node(u43) at (9,3) {$U_{43} = \{ -\frac{3}{2}, \{-1|-2\}, \{-1*|-2\}\}$};
\node(u44) at (9,1.5) {$U_{44} = \{-2, -2*\}$};
\node(u45) at (9,0) {$U_{45} = \{-3\}$};

\end{tikzpicture}
\end{center}
\caption{$1474$ games born by day $3$ and their stratification}
\label{u_3}
\end{figure*}

\begin{figure}[tb]
\begin{tikzpicture}[auto]

\node (j) at (0, 4.5) {$1 \uparrow$};
\node (k) at (1.5,4.5) {$1 \uparrow *$};
\node (l) at (3,4.5) {$\{2 | 1, 1*\}$};
\node (m) at (4.5,4.5) {$\{2 | \{1 | *\}\}$};
\node (n) at (6,4.5) {$\{2 | \{ 1 | 0 \}\}$};

\node (g) at (1.5, 6) {$\frac{3}{2}$};
\node (h) at (3, 6){$\{2|1\}$};
\node (i) at (4.5,6) {$\{2|1*\}$};

\draw[-] (j) to (g);
\draw[-] (k) to (g);
\draw[-] (k) to (h);
\draw[-] (l) to (h);
\draw[-] (j) to (i);
\draw[-] (l) to (i);
\draw[-] (m) to (i);
\draw[-] (n) to (i);

\end{tikzpicture}
\caption{$\mathcal{G}_3$}
\label{figmatch}
\end{figure}

For sets A and B, let $A \oplus B$ be $A\cup B$ and $A \cap B= \emptyset$.
Then, $\mathbb{G}_3=U_1 \oplus U_2 \oplus \cdots \oplus U_{45}$.
Here, every $U_i$ is an anti-chain by Algorithm \ref{algo1}.
The set with the highest number of elements is $U_{23}$ with $86$ elements. In addition, when $1 \leq i \leq 22$, $|U_i| < |U_{i+1}|$, and when $23 \leq i \leq 45$, $|U_i| > |U_{i + 1}|$.
Further, the following lemma was also established.

\begin{lemma}
\label{jouge}
The stratification is upper and lower symmetric. That is, for any $u \in U_i$, $-u \in U_{46-i}$ holds.
\end{lemma}

Next, for each $i$, we constructed bipartite graph $\mathcal{G}_i =(V_1\oplus V_2,E)$ as follows.
\begin{enumerate}
\item Let $|V_1| = |U_i|, |V_2| = |U_{i+1}|$, $v_{1j}$ be the vertex corresponding to $s_j \in U_i$, and $v_{2k}$ be the vertex corresponding to $t_k \in U_{i+1}$.
\item For any $s_j \in U_i, t_k \in U_{i+1}$, if $s_j > t_k$ then $(v_{1j}, v_{2k}) \in E$, and otherwise, $(v_{1j}, v_{2k}) \not \in E$.
\end{enumerate}

As an example, Figure \ref{figmatch} shows $\mathcal{G}_3$.
Here, we examined the maximum matching of each $\mathcal{G}_i$ and obtained the following results.
\begin{lemma}
\label{match}
Let $M(\mathcal{G})$ be the number of elements of a maximum matching of graph $\mathcal{G}$. Then, $M(\mathcal{G}_i) = \min(|U_i|, |U_{i+1}|)$. 
\end{lemma}
That is, any element in the set with a smaller number of elements is always included in the maximum matching.

\subsection{Chain division of $\mathbb{G}_3$}
Next, we divide $\mathbb{G}_3$ by another way.
\begin{definition}
$S=T_1 \oplus T_2 \oplus \cdots \oplus T_m$ is a {\em chain division} of $S$ by $T_1, T_2, \ldots, T_m$ if for any $i$, every $s, t \in T_i (s \neq t)$ satisfies $s>t$ or $s<t$.   
\end{definition}
For example, Figure \ref{t_2} shows a chain division of $\mathbb{G}_2$ by $T_1, T_2, T_3,$ and
 $T_4$.

\begin{figure}[tb]
\begin{tikzpicture}[auto]
\node (a) at (1.5, 0) {$-2$}; 
\node (b) at (0, 1.5) {$-1$};
\node (c) at (3, 1.5) {$-1*$};
\draw[-] (a) to (c);
\node (d) at (1.5, 3) {$-\frac{1}{2}$};
\node (e) at (3, 3) {$\{*\mid -1\}$};
\node (f) at (4.5, 3) {$\{0 \mid -1\}$};
\draw[-] (b) to (d);
\draw[-] (c) to (f);
\node (g) at (1.5, 4.5) {$\downarrow$};
\node (h) at (3, 4.5) {$\downarrow *$};
\node (i) at (4.5, 4.5) {$\{0, * \mid -1\}$};
\draw[-] (d) to (h);
\draw[-] (e) to (g);
\draw[-] (f) to (i);
\node (j) at (0, 6) {$0$};
\node (k) at (1.5, 6) {$*$};
\node (l) at (3,6) {$*2$};
\node (m) at (4.5,6) {$\pm 1$};
\draw[-] (g) to (j);
\draw[-] (h) to (k);
\draw[-] (i) to (m);
\node (n) at (1.5, 7.5) {$\uparrow$};
\node (o) at (3, 7.5) {$\uparrow*$};
\node (p) at (4.5, 7.5) {$\{1 \mid 0,*\}$};
\draw[-] (j) to (n);
\draw[-] (k) to (o);
\draw[-] (m) to (p);
\node (q) at(1.5, 9) {$\frac{1}{2}$};
\node (r) at(3,9) {$\{1 \mid * \}$};
\node (s) at (4.5,9) {$\{1 \mid 0\}$};
\draw[-] (n) to (r);
\draw[-] (o) to (q);
\draw[-] (p) to (s);
\node (t) at(0, 10.5) {$1$};
\node(u) at (3,10.5){$1*$};
\draw[-] (q) to (t);
\draw[-] (s) to (u);
\node (v) at(1.5,12) {$2$};
\draw[-] (u) to (v);
\node(t1) at(1,12.5) {$T_1$};
\node(t2) at(-0.5, 11) {$T_2$};
\node(t3) at(0.5,7) {$T_3$};
\node(t4) at(3.3,5.7) {$T_4$};
\end{tikzpicture}
\caption{Chain division of games born by day $2$}
\label{t_2}
\end{figure}

Next, we consider the chain division of $\mathbb{G}_3$. We regard every element in $\mathbb{G}_3$ as a vertex and the order of elements as edges.
Here, we consider deleting all edges except for the matching obtained from Lemma \ref{match}.
Then, from every element in $U_1, U_2, \ldots, U_{21},$ or $U_{22},$ one can reach to an element in $U_{23}$ through some edges. Further, from Lemma \ref{jouge}, there is an upper and lower symmetry and therefore, from every element in $U_{24}, U_{25}, \ldots, U_{44},$ and $U_{45},$ one can also reach to an element in $U_{23}$ through corresponding edges.
Thus, there exists as many chains as elements in $U_{23}$.

\begin{theorem}
By sets $T_1, T_2, \ldots, T_{86}$, we have a chain division
$\mathbb{G}_3 = T_1 \oplus T_2 \oplus \cdots \oplus T_{86}.$
\end{theorem}

We caluculated using programming and obtained a chain division. The length of each chain is shown in \ref{tableT}. 

We also obtained the following result using the pigeonhole principle and designating that $U_{23}$ has $86$ elements.

\begin{table}[tb]
\caption{Number of elements in $T_i$}
\label{tableT}
\begin{tabular}{r|cccccccccccccccccccccc}
\hline
\hline
$i$ & 1 & 2 & 3 & 4 & 5 & 6 & 7 & 8 & 9 & 10 & 11 & 12 & 13 & 14 & 15 & 16 & 17 & 18 & 19 & 20 & 21 & 22 \\ \hline
$|T_i|$ & 30 & 32 & 37 & 23 & 26 & 37 & 20 & 28 & 35 & 21 & 33 & 37 & 23 & 20 & 29 & 21 & 29 & 27 & 20 & 27 & 25 & 16 \\
\end{tabular}


\begin{tabular}{cccccccccccccccccccccccc}
\hline
\hline
23 & 24 & 25 & 26 & 27 & 28 & 29 & 30 & 31 & 32 & 33 & 34 & 35 & 36 & 37 & 38 & 39 & 40 & 41 & 42 & 43 & 44 & 45 & 46 \\ \hline
24 & 25 & 24 & 13 & 17 & 21 & 32 & 21 & 19 & 28 & 19 & 19 & 17 & 17 & 12 & 17 & 17 & 15 & 30 & 11 & 23 & 11 & 13 & 20\\
\end{tabular}


\begin{tabular}{cccccccccccccccccccccccc}
\hline
\hline
47 & 48 & 49 & 50 & 51 & 52 & 53 & 54 & 55 & 56 & 57 & 58 &59 & 60 & 61 & 62 & 63 & 64 & 65 & 66 & 67 & 68 & 69 & 70  \\ \hline
17 & 26 & 21 & 13 & 9 & 13 & 11 & 9 & 11 & 10  &11 & 13 & 9 & 20 & 21 & 10 & 6 & 9 & 9 & 23 & 16 & 9  & 11 & 11\\
\end{tabular}


\begin{tabular}{cccccccccccccccc}
\hline
\hline
71 & 72 & 73 & 74 & 75 & 76 & 77 & 78 & 79 & 80 & 81 & 82 & 83 & 84 & 85 & 86\\ \hline
12 & 7 & 5 & 4 & 9 & 5 & 5 & 6 & 4 & 20  & 13 & 1 & 1 & 1 & 1 & 1\\
\end{tabular}

\end{table}

\begin{lemma}
\label{lem86}
The anti-chain in $\mathbb{G}_3$ with the largest number of elements has 86 elements.
\end{lemma}

\section{Improving upper and lower bounds of $|\mathbb{G}_4|$}
\label{sec4}
\subsection{Improving of lower bounds}
\label{subsec41}
First, in order to improve the lower bound, we prepared the following lemma,
\begin{lemma}
\label{at}
Assume that $n>1$. We also assume that $S=\{s_1,s_2,\ldots,s_k\}$ is an anti-chain of games born by day $n$. That is, for any $i,j$, $s_i, s_j \in \mathbb{G}_n$ and $s_i \not\lessgtr s_j$ holds. Then, $\{n \mid s_1, s_2, \ldots, s_k\}, \{s_1, s_2, \ldots, s_k\mid -n\}, \{n-1 \mid s_1, s_2, \ldots, s_k\}, $ and $\{s_1, s_2, \ldots, s_k\mid -(n-1)\}\in \mathbb{G}_{n+1}$ are canonical forms.

\end{lemma}

\begin{proof}
Consider $\{n - 1 \mid s_1, s_2, \ldots, s_k\}\in \mathbb{G}_{n+1}$.
If this is not a canonical form, then there is a dominated or reversible option in the left or right options.


First, there is only one left option, $n-1$, which is not a dominated option.
The right options also have no dominated option because the set of all right options is an anti-chain. 
Further, the left option $n-1$ has no right option; therefore, this is not a reversible option.

Intuitively, the remaining case is that a right option, $s_i$, is a reversible option.
If $s_i$ is a reversible option, then $s_i$ has a left option $s_i^L$ and it satisfies $\{n - 1 \mid s_1, s_2, \ldots, s_k\}\leq s_i^L$. However, considering 
$\{n - 1 \mid s_1, s_2, \ldots, s_k\} - s_i^L$, this game has a left option $n - 1 - s_i^L$. It is known that for any game where $g$ is born by day $n$, $n \geq g$ holds. As $s_i$ is born by day $n$, $s_i^L$ is born by day $n-1$. Thus, $n - 1 \geq s_i^L$. This yields $n - 1 - s_i^L \geq 0$ and $\{n - 1 \mid s_1, s_2, \ldots, s_k\} - s_i^L \mid> 0$, which means $\{n - 1 \mid s_1, s_2, \ldots, s_k\}\leq s_i^L$ does not happen. Other cases are proved in similar way.
\end{proof}

From lemmas \ref{lem86} and \ref{at}, we immediately obtain the following result. 
\begin{corollary}
\label{c1}
$4 \cdot 2^{86} = 2^{88} \leq |\mathbb{G}_4|.$
\end{corollary}
As $2 > 10^{0.3}$, $10^{26.4} < 2^{88} \leq |\mathbb{G}_4|$, then we have successed in significantly improving the previously known lower bound $3 \cdot 10^{12}$.


\subsection{Further improvement of the lower bound}
\label{subsec42}
This result can be improved by scrutinizing the number of anti-chains.
For any element $u$ in $U_{22}$, we checked the number of elements that are smaller than $u$ and included in $U_{23}$. Then there are $9, 25, 33, 14$ elements which have $2, 3, 4, 5$ smaller elements in $U_{23}$ respectively.
Therefore, the number of anti-chains which have one element from $U_{22}$ and some (possibly 0 as well) elements from $U_{23}$ is $9 \cdot 2^{84} + 25 \cdot 2^{83} + 33 \cdot 2^{82} + 14 \cdot 2^{81}= (9 \cdot 8 + 25 \cdot 4 + 33 \cdot 2 +14)2^{81} = 252  \cdot 2^{81}$.



From lemma \ref{jouge}, $U_{22}$ and $U_{24}$ have a sign-reversed relationship of elements and for every element in $U_{23}$, its inverse is also in $U_{23}$. Therefore, the number of anti-chains with one element from $U_{24}$ and some (possibly 0 as well) elements from $U_{23}$ is the same.


We also consider the number of anti-chains which have two elements from $U_{22}$ and some (possibly 0 as well) elements from $U_{23}$. It is at least $\frac{9 \cdot 8}{2} \cdot 2^{82} + 9 \cdot 25 \cdot 2^{81} + \frac{25 \cdot 24}{2} \cdot 2^{80} + 9 \cdot 33 \cdot 2^{80} + 9 \cdot 14 \cdot 2^{79} + 25 \cdot 33 \cdot 2^{79} + 25\cdot 14 \cdot 2^{78} + \frac{33 \cdot 32}{2} \cdot 2^{ 78} + 33\cdot 14 \cdot 2^{77} + \frac{14\cdot 13}{2}\cdot 2^{76}$. By calculating, this is $(2304 + 7200 + 4800 + 4752 + 1008 + 6600 + 1400 + 2112 + 924 + 91) \cdot 2^{76} = 31191 \cdot 2^{76}$.


We can also use upper and lower symmetry. Therefore, calculating the total number obtained so far yields $2 ^ {86} + 2 \cdot 252 \cdot 2^{81} + 2 \cdot 31191 \cdot 2^{76} = (2^9 + 252 \cdot 2^5 + 31191) \cdot 2^{77} = (512 + 8064 + 31191) \cdot 2^{77} = 39767 \cdot 2^{77} > 2^{15} \cdot 2^{77} = 2^{92}$. Thus, the following holds.
\begin{corollary}
\label{c1_2}
$4 \cdot 2^{92} = 2^{94} < |\mathbb{G}_4|.$
\end{corollary}
Thus,  $10^{28.2}< |\mathbb{G}_4|$.

\subsection{Improving of the upper bound}
\label{subsec43}
Next, we consider the upper bound. A canonical form has two sets of left and right options as anti-chains. Therefore, by using chain division, the following lemma  holds. 
\begin{lemma}
\label{lem431}
For any game $G \in \mathbb{G}_4$, let the set of games in left options of $G$ be $S=\{s_1, s_2, \ldots, s_m\}$. Then, for any $i,j(i \neq j)$, if $s_i \in T_x, s_j \in T_y$ then $x \neq y$. This is also true for the set of games in right options.
\end{lemma}
\begin{proof}
If $s_i \in T_x$ and $s_j \in T_y,$ then $s_i < s_j$ or $s_i > s_j$ holds, which is a contradiction.
\end{proof}
From this lemma, the following holds.
\begin{lemma}
$|\mathbb{G}_4|\leq ((|T_1| + 1) \times (|T_2| + 1) \times \cdots \times (|T_{86}| + 1))^2$
\end{lemma}
\begin{proof}
From Lemma \ref{lem431}, in the set of left options of a canonical form born by day $4$, there is at most one element in $|T_i|$ for each $1 \leq i \leq 86$. This also holds in the set of right options. Therefore,  $|\mathbb{G}_4|\leq ((|T_1| + 1) \times (|T_2| + 1) \times \cdots \times (|T_{86}| + 1))^2$.
\end{proof}
By calculating this value, we have about $3.7979\times 10^{202}$. Therefore, we have the following result.
\begin{corollary}
\label{c2}
$|\mathbb{G}_4|< 3.8 \cdot 10^{202}.$
\end{corollary}

\subsection{Further improvement of the upper bound}
\label{subsec44}
We also consider the further improvement of this upper bound.
For each element in $T_i$, count incompareble elements in $T_j(i<j)$ and let $t_{i,j}$ be their maximum. That is, $t_{i,j} = \max(|\{t' \in T_{j} \mid t \not \lessgtr t'\}|)_{t \in T_i}$.
In addition, let $S_i = |T_i| \times (t_{i,i+1} + 1) \times (t_{i,i + 2} + 1) \times \cdots \times (t_{i,86} + 1)$. Then, $S_i$ is an upper bound of the number of anti-chains that do not include any elements in $T_k(k<i)$ and include an element in $T_i$. Therefore, $|\mathbb{G}_4| \leq (\sum_{i = 1}^{86} S_i + 1) ^ 2$.

By calculating this, we obtained the following corollary. The values of $S_i$ are shown in \ref{tables}. 

Since $\sum_{i=k}^{86}S_i \leq (|T_k| + 1) \times (|T_{k+1}| + 1) \times \cdots \times (|T_{86}| + 1)$, and since either $\sum_{i=k}^{86}S_i$ and $(|T_k| + 1) \times (|T_{k+1}| + 1) \times \cdots \times (|T_{86}| + 1)$ has little effect on the upper bound as $i$ increases, we obtained values of $S_i$ up to $S_7$ and took the upper bound as $|\mathbb{G}_4| \leq (\sum_{i = 1}^{7} S_i + (|T_8| + 1) \times (|T_{9}| + 1) \times \cdots \times (|T_{86}| + 1))$. Here, $(|T_8| + 1) \times (|T_{9}| + 1) \times \cdots \times (|T_{86}| + 1))<1.0 \cdot 10^{91}$.


\begin{table}[tb]
\caption{Value of $S_i$}
\label{tables}
\begin{center}
\begin{tabular}{r|c}
\hline \hline

$S_1$ & $<4.0 \cdot 10^{90}$ \\ 
$S_2$ & $<1.8 \cdot 10^{92}$ \\ 
$S_3$ & $<1.5 \cdot 10^{89}$ \\ 
$S_4$ & $<1.5 \cdot 10^{87}$ \\ 
$S_5$ & $<5.5 \cdot 10^{86}$ \\ 
$S_6$ & $<3.0 \cdot 10^{87}$ \\ 
$S_7$ & $<3.0 \cdot 10^{83}$ \\ 

\end{tabular}
\end{center}
\end{table}

\begin{corollary}
\label{c2_2}
$|\mathbb{G}_4| < 4.0 \cdot 10^{184}.$ 
\end{corollary}

Therefore, from Corollaries \ref{c1_2} and \ref{c2_2}, we can obtain the following theorem, significantly improving	 previously known upper and lower bounds. 

\begin{theorem}
$10^{28.2} < |\mathbb{G}_4|< 4.0 \cdot 10^{184}.$
\end{theorem}

\section{Verification}
\label{sec5}
Although this study has achieved significant improvements in the upper and lower bounds of the number of games born by day $4$, there is still a large gap.
For the number of games born by day $3$ (already known as the true value), we used the same method of this study to explore which one of them is closer to the true value. In other words, we examined which of the obtained upper and lower bounds was closer to the true value.

As shown in Figure \ref{u_2}, the largest set of the stratification has $4$ elements. Therefore, the lower bounds obtained using the method in Sections \ref{subsec41} and  \ref{subsec42} are $2^4 \cdot 4 = 64$ and $4 \cdot ( 2^4 +  2 \cdot 12 + 2 \cdot 6) = 208$, respectively.

Meanwhile, Figure \ref{t_2} shows a chain division of games born by day $2$. There are four chains and the numbers of elements are $9, 7, 5,$ and $1$. Therefore, the upper bounds obtained using the method in Sections \ref{subsec43} and  \ref{subsec44} are $(10\cdot 8 \cdot 6 \cdot 2) ^ 2 = 921600$ and $(9 \cdot 8 \cdot 4 \cdot 2 + 7 \cdot 6 \cdot 2 + 5 \cdot 2 + 1 + 1) ^ 2= 451584$, respectively.

The true value, $1474$, is close to the lower bound. Therefore, if the same can be said for $\mathbb{G}_4$, the lower bound $10^{28.2}$ is closer to the true value of the number of games born by day $4$ than the upper bound $4.0 \cdot 10^{184}.$

\section{Conclusion}
\label{sec6}
In this study, we significantly improved the upper and lower bounds of the number of games born by day $4$. We obtained this result by using some algebraic properties of combinatorial game theory.
As the improvement of upper and lower bounds has long remained unsolved, this result is an important development for all aspects of the game as an algebraic object and aspects analyzing the game itself.
However, there remains a gap in the width of the upper and lower bounds and we will continue trying to improve them.
Specifically, we improved the upper and lower bounds of $|\mathbb{G}_4|$, by calculating the length of chains, the number of chains, and the lower bounds of the number of anti-chains of $\mathbb{G}_3$. As this method can be generalized, by calculating the length of chains, the number of chains, and the lower bounds of the number of anti-chains of $\mathbb{G}_n$, we also try improving the upper and lower bounds of $|\mathbb{G}_{n+1}|$.
This research is also related to counting and algorithms, and we will continue to contribute to combinatorial game theory and related fields by improving our methods and conducting research on applications of the results obtained.

\appendix
\section{$\mathbb{G}_3$ and its stratification}

Tables \ref{tablea1}, \ref{tablea2}, and \ref{tablea3} shows every element in $\mathbb{G}_3$ and its stratification. As $U_{24}, \ldots, U_{45}$ can be constructed by using upper and lower symmetry, we omit them. The $i$-th element from the left in each row belongs $T_i$ in the chain division we used.
An Excel file integrating these tables is available at: https://sites.google.com/site/kokisuetsugu2/games-born-by-day-3

\begin{table}[b]
\caption{Elements in $U_i( 1 \leq i \leq 10)$}
\label{tablea1}
\small
\begin{tabular}{c|l}
\hline \hline
$U_1$ & \begin{tabular}{l} $3$ \end{tabular} \\ \hline
$U_2$ & \begin{tabular}{l} $2,2*$  \end{tabular} \\ \hline
$U_3$ & \begin{tabular}{l} $\frac{3}{2},\{2|1\},\{2|1*\}$ \end{tabular} \\ \hline
$U_4$ & \begin{tabular}{l} $1\uparrow,1\uparrow*,\{2|1,1*\},\{2||1|*\},\{2||1|0\}$ \end{tabular} \\ \hline
$U_5$ & \begin{tabular}{l} $1*2,1*,\{2|\frac{1}{2}\},1+\text{Tiny}_{1*},1+\text{Tiny}_1,\{2|1,\{1|*\}\},\{2|1,\{1|0\}\},$ \\ $\{2|\{1|0\},\{1|*\}\}$ \end{tabular} \\ \hline

$U_6$ & \begin{tabular}{l} $\{1,1*|\frac{1}{2}\},1\downarrow*,\{2|\frac{1}{2},\{1|*\}\},\{1,1*|1,\{1|*\}\},\{1,1*|1,\{1|0\}\},$ \\
$\{2|\frac{1}{2},\{1|*\}\},\{2|\frac{1}{2},\{1|0\}\},\{2|1,\{1|0\},\{1|*\}\},\{2||1|0,*\}$ \end{tabular} \\ \hline

$U_7$ & \begin{tabular}{l} $\{1,1*|\frac{1}{2},\{1|*\}\},\{1|\frac{1}{2}\},\{2|\frac{1}{2},\{1|0\},\{1|*\}\},\{1|1,\{1|*\}\},$ \\

$\{1|1,\{1|0\}\}, \{2|\uparrow\},\{1,1*|\frac{1}{2},\{1|0\}\},\{1,1*|1,\{1|0\},\{1|*\}\},$ \\

$\{1||1|0,*\},\{2|\uparrow*\},\{2|1,\{1|0,*\}\},\{2|\pm1\}$ \end{tabular} \\ \hline

$U_8$ & \begin{tabular}{l}
$\{1,1*|\uparrow\},\{1|\frac{1}{2},\{1|*\}\},\{2|\uparrow,\{1|0\}\},\{1|1,\{1|0\},\{1|*\}\},$ \\
$\{1|\frac{1}{2},\{1|0\}\},\{2|0\},\{1,1*|\uparrow*\},\{1,1*|\frac{1}{2},\{1|0\},\{1|*\}\},$ \\
$\{1,1*|1,\{1|0,*\}\}, \{2|\uparrow*,\{1|*\}\},\{2|\frac{1}{2},\{1|0,*\}\},1+\text{Tiny}_ 2,$ \\
$\{2|1,\pm1\},\{2|*\}$\end{tabular}
 \\ \hline

$U_9$ & \begin{tabular}{l} $\{1,1*|0\},\{1|\uparrow\},\{2|\uparrow,\{1|0,*\}\},\{1|\frac{1}{2},\{1|0\},\{1|*\}\},\{1|\uparrow*\},$ \\

$\{2|0,\{1|0\}\},\{1,1*|\uparrow*,\{1|*\}\},\{1,1*|\uparrow,\{1|0\}\},\{1|1,\{1|0,*\}\},$ \\

$\{2|\uparrow, \uparrow*\},\{1,1*|\frac{1}{2},\{1|0,*\}\},1,\{1,1*|1,\pm1\},\{1,1*|*\},$ \\

$\{2|\frac{1}{2},\pm1\},\{2|\uparrow*,\{1|0,*\}\},\{2|*,\{1|*\}\}$ \end{tabular} \\ \hline

$U_{10}$ & \begin{tabular}{l} $\{1,1*|0,\{1|0\}\},\{1,\{1|0\}|0\},\{2|0,\{1|0,*\}\},\{1|\frac{1}{2},\{1|0,*\}\},$ \\

$\{1|\uparrow*,\{1|*\}\},\{2|0,\uparrow*\},\{1,1*|\uparrow, \uparrow*\},\{1|\uparrow,\{1|0\}\},\{1|1,\pm1\},$\\

$\{2|\uparrow, \uparrow*,\{1|0,*\}\},\{1,1*|\uparrow,\{1|0,*\}\},1\downarrow,\{1,1*|\frac{1}{2},\pm1\},$ \\

$\{1,\{1|*\}|*\},\{2|\uparrow,\pm1\},\{1,1*|\uparrow*,\{1|0,*\}\},\{1,1*|*,\{1|*\}\},$ \\ 

$\{2|*,\uparrow\},\{2|\uparrow*,\pm1\},\{2|*,\{1|0,*\}\}$ \end{tabular} \\ \hline

\end{tabular}
\end{table}
\begin{table*}[tb]
\caption{Elements in $U_i( 11 \leq i \leq 19)$}
\label{tablea2}
\small
\begin{tabular}{c|l}
\hline \hline

$U_{11}$ & \begin{tabular}{l} $\{1,\{1|0\}|0,\{1|0\}\},\{1|0\},\{2|0,\uparrow*,\{1|0,*\}\},\{1|\frac{1}{2},\pm1\},\{1|\uparrow, \uparrow*\},\{1,1*|0,\uparrow*\},\{1,1*|\uparrow, \uparrow*,\{1|0,*\}\},\{1|\uparrow,\{1|0,*\}\},\frac{3}{4},\{2|\uparrow, \uparrow*,\pm1\},$ \\
$\{1,1*|0,\{1|0,*\}\},\{1*|\frac{1}{2}\},\{1,1*|\uparrow,\pm1\},\{1|*\},\{2|0,\pm1\},\{1|\uparrow*,\{1|0,*\}\},\{1,\{1|*\}|*,\{1|*\}\},\{1,1*|*,\uparrow\},\{1,1*|\uparrow*,\pm1\},$ \\

$\{1,1*|*,\{1|0,*\}\},\{2|0,*\},\{2|*,\uparrow,\{1|0,*\}\},\{2|*2\},\{2|*,\pm1\}$ \end{tabular} \\ \hline

$U_{12}$ & \begin{tabular}{l} $\{1,\{1|0\}|0,\{1|0,*\}\},\{1|0,\{1|0\}\},\{2|0,*2\},\frac{1}{2}*,\{1|\uparrow,\uparrow*,\{1|0,*\}\},\{1,\{1|0\}|0,\uparrow*\},\{1,1*|0,\uparrow*,\{1|0,*\}\},\{1|\uparrow,\pm1\},\frac{1}{2},\{2|*,\uparrow,\pm1\},$ \\
$\{1,1*|0,\pm1\},\{1*|\uparrow*\},\{1*|\uparrow\},\{1|*,\{1|*\}\},\{2|0,\uparrow*,\pm1\},\{1|\uparrow*,\pm1\},\{1,\{1|*\}|*,\uparrow\},\{1,1*|0,*\},\{1,1*|\uparrow,\uparrow*,\pm1\},$ \\
$\{1,\{1|*\}|*,\{1|0,*\}\},\{2|0,*,\{1|0,*\}\},\{1,1*|*,\uparrow,\{1|0,*\}\},\{1,1*|*2\},\{1,1*|*,\pm1\},\{2|\pm1,*2\},\{2|*,*2\}$ \end{tabular} \\ \hline

$U_{13}$ & \begin{tabular}{l} $\{1|0,\{1|0,*\}\},\{1|0,\uparrow*\},\{2|0,*,*2\},\{\frac{1}{2}|\uparrow\},\{1|*2\},\{1,\{1|0\},\{1|*\}|0,*\},\{1,\{1|0\}|0,\uparrow*,\{1|0,*\}\},\{1|\uparrow, \uparrow*,\pm1\},\frac{1}{4},\{2|*,\pm1,*2\},$ \\

$\{1,\{1|0\}|0,\pm1\},\{1*|\uparrow, \uparrow*\},\{1*|0\},\{1|*,\uparrow\},\{2|0,\pm1,*2\},\{\frac{1}{2}|\uparrow*\},\{1,\{1|*\}|*,\uparrow,\{1|0,*\}\},\{1,1*|0,*,\{1|0,*\}\},\{1,1*|0,\uparrow*,\pm1\},$ \\ 

$\{1|*,\{1|0,*\}\},\{2|0,*,\pm1\},\{1,1*|*,\uparrow,\pm1\},\{1,1*|0,*2\},\{1,\{1|*\}|*,\pm1\},\{1,1*|\pm1,*2\},\{1,1*|*,*2\},\{2|\downarrow\},\{1*|*\},$ \\
$\{2||0,*|-1\},\{2|\downarrow*\}$ \end{tabular} \\ \hline

$U_{14}$ & \begin{tabular}{l} $\{1|0,\pm1\},\{1|0,\uparrow*,\{1|0,*\}\},\{2|0,\downarrow*\},\{\frac{1}{2}|\uparrow, \uparrow*\},\{1,\{1|*\}|*,*2\},\{1,\{1|*\}|0,*\},\{1,\{1|0\}|0,*2\},\{1|\pm1,*2\},\Uparrow*,\{2|*,\{0,*|-1\}\},$ \\ 
$\{\frac{1}{2},\{1|0\}|0\},\{1*|*,\uparrow\},\{1*|0,\uparrow*\},\{1,\{1|0\}|0,*\},\{2|0,\{0,*|-1\}\},\Uparrow,\{1,\{1|*\}|*,\uparrow,\pm1\},\{1,\{1|0\},\{1|*\}|0,*,\{1|0,*\}\},$ \\

$\{1,\{1|0\}|0,\uparrow*,\pm1\},\{1|*,\uparrow,\{1|0,*\}\},\{2|0,*,\pm1,*2\},\{1,1*|0,*,\pm1\},\{1,1*|0,\pm1,*2\},\{1|*,\pm1\},\{1,1*|*,\pm1,*2\},\{1,1*|0,*,*2\},$ \\ 

$\{1,1*|\downarrow\},\{\frac{1}{2},\{1|*\}|*\},\{1,1*||0,*|-1\},\{1,1*|\downarrow*\},\{2|\downarrow,\pm1\},\{2|*,\downarrow\},\{1*|*2\},\{2|\downarrow*,\pm1\}$ \end{tabular} \\ \hline

$U_{15}$ & \begin{tabular}{l} $\{\frac{1}{2}|0\},\{1|0,\uparrow*,\pm1\},\{1,1*|0,\downarrow*\},\{\frac{1}{2},\{1|*\}|*,\uparrow\},\{1,\{1|*\}|*,\pm1,*2\},\{1,\{1|0,*\}|0,*\},\{1|0,*2\},\{1||0,*|-1\},\{0|\uparrow, \uparrow*\},$ \\
$\{2|\downarrow*,\{0,*|-1\}\},\{0,\{1|0\}|0\},\{1*|*,*2\},\{1*|0,*2\},\{1,\{1|0\}|0,*,\{1|0,*\}\},\{2|0,*,\{0,*|-1\}\},\{1|*||*\},\{1,\{1|0\},\{1|*\}|0,*,\pm1\},$ \\
$\{1,\{1|*\}|0,*,\{1|0,*\}\},\{\frac{1}{2},\{1|0\}|0,\uparrow*\},\{1|*,*2\},\{1,1*|0,*,\pm1,*2\},\{1*|0,*\},\{1,\{1|0\}|0,\pm1,*2\},\{1|*,\uparrow,\pm1\},\{1,1*|*,\{0,*|-1\}\},$ \\
$\{1,\{1|0\},\{1|*\}|0,*,*2\},\{1,\{1|0\}|\downarrow\},\{\frac{1}{2}|*\},\{1,1*|0,\{0,*|-1\}\},\{1,\{1|*\}|\downarrow*\},\{1,1*|\downarrow,\pm1\},\{1,1*|*,\downarrow\},\{\frac{1}{2}|*2\},$ \\
$\{1,1*|\downarrow*,\pm1\},\{2|0,\downarrow*,\pm1\},\{2|\downarrow,\{0,*|-1\}\},\{2|*,\downarrow,\pm1\},\{2|\downarrow, \downarrow*\},\{1*||0,*|-1\}$ \end{tabular} \\ \hline

$U_{16}$ & \begin{tabular}{l} $\{\frac{1}{2}|0,\uparrow*\},\{1|0,\pm1,*2\},\{1,\{1|0\},\{1|*\}|0,\downarrow*\},\{\frac{1}{2}|*,\uparrow\},\{\frac{1}{2},\{1|*\}|*,*2\},\{1|0,*\},\{1|\downarrow\},\{1,\{1|*\}|*,\{0,*|-1\}\},\{0,\{1|0\}|0,\uparrow*\},$\\

$\{2||0|-1\}, \uparrow^{[2]}+*,\{1*|*,\{0,*|-1\}\},\{1*|0,\{0,*|-1\}\},\{1,\{1|0,*\}|0,*,\{1|0,*\}\},\{1,1*|0,*,\{0,*|-1\}\},\{1|*||*,\uparrow\},$ \\

$\{1,\{1|0\}|0,*,\pm1\}, \{1,\{1|*\}|0,*,\pm1\},\{\frac{1}{2},\{1|0\},\{1|*\}|0,*\},\{1|\downarrow*\},\{1,\{1|0\},\{1|*\}|0,*,\pm1,*2\},\{1*|0,*,*2\},\{\frac{1}{2},\{1|0\}|0,*2\},$ \\

$\{1|*,\pm1,*2\},\{1,1*|\downarrow*,\{0,*|-1\}\},\{1,\{1|*\}|0,*,*2\},\{1,\{1|0\}|\downarrow,\pm1\},\uparrow^{[2]},\{1,\{1|0\}|0,\{0,*|-1\}\},\{1,\{1|*\}|\downarrow*,\pm1\}, $ \\

$\{1,1*|\downarrow,\{0,*|-1\}\},\{1,\{1|0\},\{1|*\}|*,\downarrow\},\uparrow*3,\{1,1*|0,\downarrow*,\pm1\},\{2|0,\downarrow*,\{0,*|-1\}\},\{2|*,\downarrow,\{0,*|-1\}\},\{1,1*|*,\downarrow,\pm1\},$ \\

$\{1,1*|\downarrow, \downarrow*\},\{\frac{1}{2}||0,*|-1\},\{1*|\downarrow\},\{1*|\downarrow*\},\{2|\downarrow, \downarrow*,\pm1\},\{2||*|-1\}, \{2|-\frac{1}{2}\},\{1,\{1|0\}|0,*,*2\}$ \end{tabular} \\ \hline

$U_{17}$ & \begin{tabular}{l} $\{0,\uparrow*|0,\uparrow*\},\{\frac{1}{2}|0,*2\},\{1,\{1|0\},\{1|*\}|0,\downarrow*,\pm1\},\{\uparrow|*,\uparrow\},\{\frac{1}{2},\{1|*\}|\downarrow*\},\{1|0,*,\{1|0,*\}\},\{1|\downarrow,\pm1\},\{1|*,\{0,*|-1\}\},$ \\

$\{\{1|0\},\{1|*\}|0,*\},\{1,1*||0|-1\},\uparrow*,\{1*|0,*,\{0,*|-1\}\},\{\frac{1}{2},\{1|0\}|0,\{0,*|-1\}\},\{1,\{1|0,*\}|0,*,\pm1\},$ \\

$\{1,\{1|0\},\{1|*\}|0,*,\{0,*|-1\}\},\{1|*||*,*2\},\{\frac{1}{2},\{1|0\}|0,*\},\{1,\{1|*\}|0,*,\pm1,*2\},\{\frac{1}{2},\{1|*\}|0,*\},\{1,\{1|0\}|0,\downarrow*\},$ \\

$\{1,\{1|0\},\{1|*\}|*,\downarrow,\pm1\},\{\frac{1}{2},\{1|0\},\{1|*\}|0,*,*2\}, \{0,\{1|0\}|0,*2\},\{\frac{1}{2}|*,*2\},\{1,\{1|*\}|\downarrow*,\{0,*|-1\}\},\{1,\{1|*\}|0,\downarrow*\},\{\frac{1}{2},\{1|0\}|\downarrow\},$ \\

$\uparrow,\{1|0,\{0,*|-1\}\},\{1|\downarrow*,\pm1\},\{1,\{1|0\}|\downarrow,\{0,*|-1\}\},\{1,\{1|*\}|*,\downarrow\},\{0||0,*|-1\},\{1,1*|0,\downarrow*,\{0,*|-1\}\},\{2|0,\{0|-1\}\},$ \\

$\{2|\downarrow, \downarrow*,\{0,*|-1\}\}, \{1,1*|*,\downarrow,\{0,*|-1\}\},\{1,\{1|0\},\{1|*\}|\downarrow, \downarrow*\},\{\frac{1}{2},\{1|*\}|*,\{0,*|-1\}\},\{1*|\downarrow,\{0,*|-1\}\},\{1*|0,\downarrow*\},$ \\

$\{1,1*|\downarrow, \downarrow*,\pm1\},\{1,1*||*|-1\},\{1,1*|-\frac{1}{2}\},\{1,\{1|0,*\}|0,*,*2\},\{1*|*,\downarrow\},\{1*|\downarrow*,\{0,*|-1\}\},\{1,\{1|0\}|*,\downarrow\},$ \\

$\{1,\{1|0\}|0,*,\pm1,*2\},\{2|*,\{*|-1\}\},\{2|\pm1,-\frac{1}{2}\},\{2|-1\}$ \end{tabular} \\ \hline

$U_{18}$ & \begin{tabular}{l} $\{\uparrow*,\{1|*\}|0,*\},\{0,\uparrow*|0,*2\},\{\frac{1}{2},\{1|0\},\{1|*\}|0,\downarrow*\},\{\uparrow,\{1|0\}|0,*\},\{1|*||\downarrow*\},\{1|0,*,*2\},\{\frac{1}{2}|\downarrow\},\{1,\{1|0\}|0,*,\{0,*|-1\}\},$ \\

$\{\{1|0\},\{1|*\}|0,*,*2\},\{1,\{1|*\}||0|-1\},\{0,*|0,\uparrow*\},\{1*|*,\downarrow,\{0,*|-1\}\},\{\frac{1}{2}|0,\{0,*|-1\}\},\{1|0,*,\pm1\},\{\frac{1}{2},\{1|0\},\{1|*\}|0,*,\{0,*|-1\}\},$ \\

$\{\uparrow|*,*2\},\{\frac{1}{2},\{1|0\}|0,*,*2\},\{1,\{1|*\}|0,*,\{0,*|-1\}\},\{\frac{1}{2},\{1|0,*\}|0,*\},\{1,\{1|0\}|0,\downarrow*,\pm1\},\{1,\{1|*\}|*,\downarrow,\pm1\},\{\frac{1}{2},\{1|*\}|0,*,*2\},$ \\

$\{0,\{1|0\}|0,\{0,*|-1\}\},\{\frac{1}{2}|*,\{0,*|-1\}\},\{1|\downarrow*,\{0,*|-1\}\},\{1,\{1|*\}|0,\downarrow*,\pm1\},\{0,\{1|0\}|\downarrow\},\{0|*,\uparrow\},\{1|\downarrow,\{0,*|-1\}\},$ \\

$\{\frac{1}{2}|\downarrow*\},\{\frac{1}{2},\{1|0\}|\downarrow,\{0,*|-1\}\},\{1,\{1|*\}|\downarrow, \downarrow*\},\{1|*||*,\{0,*|-1\}\},\{1,\{1|0\},\{1|*\}|0,\downarrow*,\{0,*|-1\}\},\{1,1*|0,\{0|-1\}\},$ \\

$\{2|\downarrow,\{0|-1\}\},\{1,\{1|0\},\{1|*\}|*,\downarrow,\{0,*|-1\}\},\{1,\{1|0\},\{1|*\}|\downarrow, \downarrow*,\pm1\},\{\frac{1}{2},\{1|*\}|\downarrow*,\{0,*|-1\}\},\{1*||*|-1\},$ \\

$\{1*|\downarrow, \downarrow*\},\{1,1*|\downarrow, \downarrow*,\{0,*|-1\}\},\{1,\{1|0\}||*|-1\},\{1,\{1|0\},\{1|*\}|-\frac{1}{2}\},\{1,\{1|0,*\}|0,\downarrow*\},\{\frac{1}{2},\{1|0\},\{1|*\}|*,\downarrow\},$ \\

$\{1*|0,\downarrow*,\{0,*|-1\}\},\{1,\{1|0,*\}|*,\downarrow\},\{1,\{1|0,*\}|0,*,\pm1,*2\},\{1,1*|*,\{*|-1\}\},\{1,1*|\pm1,-\frac{1}{2}\},\{1,1*|-1\},\{1,\{1|0\}|*,\downarrow,\pm1\},$ \\

$\{1*||0|-1\},\{1,\{1|0\}|\downarrow, \downarrow*\},\{2|\downarrow*,\{*|-1\}\},\{2|-\frac{1}{2},\{0,*|-1\}\},\{2|-1,\pm1\}$ \end{tabular} \\ \hline

$U_{19}$ & \begin{tabular}{l} $\{\uparrow*,\{1|*\}|0,*,*2\},\{0,*|0,*2\},\{\frac{1}{2},\{1|*\}|0,\downarrow*\},\{\uparrow, \uparrow*,\{1|0,*\}|0,*\},\{\{1|0\},\{1|*\}|0,\downarrow*\},\{1|*,\downarrow\},\{0,\uparrow*|\downarrow\},\{1,\{1|0\}|0,\downarrow*,\{0,*|-1\}\},$ \\

$\{\uparrow,\{1|0\}|0,*,*2\},\{1||0|-1\},\{*,\{1|*\}|0,*\},\{\frac{1}{2},\{1|0\},\{1|*\}|*,\downarrow,\{0,*|-1\}\},\{0,\uparrow*|0,\{0,*|-1\}\},\{1|0,*,\pm1,*2\},$ \\

$\{\frac{1}{2},\{1|0\}|0,*,\{0,*|-1\}\},\{0|*,*2\},\{\frac{1}{2},\{1|0,*\}|0,*,*2\},\{1,\{1|0,*\}|0,*,\{0,*|-1\}\},\{\frac{1}{2}|0,*\},\{\frac{1}{2},\{1|0\}|0,\downarrow*\},\{\frac{1}{2},\{1|*\}|*,\downarrow\},$ \\

$\{\frac{1}{2},\{1|*\}|0,*,\{0,*|-1\}\},\{0,\{1|0\}|\downarrow,\{0,*|-1\}\},\{\uparrow|*,\{0,*|-1\}\},\{\frac{1}{2}|\downarrow*,\{0,*|-1\}\},\{1,\{1|*\}|0,\downarrow*,\{0,*|-1\}\},$ \\

$\{\{1|0\},\{1|*\}|*,\downarrow\},\{0,\{1|0\}|0,*\},\{1||*|-1\},\{\uparrow|\downarrow*\},\{\frac{1}{2}|\downarrow,\{0,*|-1\}\},\{1,\{1|*\}|\downarrow, \downarrow*,\pm1\},\{\{1|0\},\{1|*\}|0,*,\{0,*|-1\}\},$ \\

$\{1,\{1|0\},\{1|*\}|0,\{0|-1\}\},\{1,1*|\downarrow,\{0|-1\}\},\{2|\{0|-1\},\{*|-1\}\},\{1,\{1|*\}|*,\downarrow,\{0,*|-1\}\},\{1,\{1|0\}|\downarrow, \downarrow*,\pm1\},$ \\

$\{1|*||\downarrow*,\{0,*|-1\}\},\{1*|*,\{*|-1\}\},\{1*|\downarrow, \downarrow*,\{0,*|-1\}\},\{1,\{1|0\},\{1|*\}|\downarrow, \downarrow*,\{0,*|-1\}\},\{\frac{1}{2},\{1|0\}||*|-1\},\{1,\{1|*\}|-\frac{1}{2}\},$ \\

$\{1|0,\downarrow*\},\{\frac{1}{2},\{1|0\},\{1|*\}|\downarrow, \downarrow*\},\{\frac{1}{2},\{1|0\},\{1|*\}|0,\downarrow*,\{0,*|-1\}\},\{1,\{1|0,*\}|*,\downarrow,\pm1\},\{1,\{1|0,*\}|0,\downarrow*,\pm1\},$ \\

$\{1,\{1|0\},\{1|*\}|*,\{*|-1\}\},\{1,\{1|0\},\{1|*\}|\pm1,-\frac{1}{2}\},\{1,\{1|0\},\{1|*\}|-1\},\{\frac{1}{2},\{1|0\}|*,\downarrow\},\{\frac{1}{2},\{1|*\}||0|-1\},\{1,\{1|0,*\}|\downarrow, \downarrow*\},$ \\ 

$\{1,1*|\downarrow*,\{*|-1\}\},\{1,1*|-\frac{1}{2},\{0,*|-1\}\},\{1,1*|-1,\pm1\},\{1*|0,\{0|-1\}\},\{1,\{1|0\}|-\frac{1}{2}\},\{1*|-\frac{1}{2}\},\{2|-\frac{1}{2},\{*|-1\}\},$ \\

$\{1,\{1|0\}|*,\downarrow,\{0,*|-1\}\},\{2|-1,\{0,*|-1\}\},\{2|\{0|-1\},-\frac{1}{2}\}$ \end{tabular} \\ \hline

\end{tabular}
\end{table*}

\begin{table*}[tb]
\caption{Elements in $U_i( 20 \leq i \leq 23)$}
\label{tablea3}
\small
\begin{tabular}{c|l}
\hline \hline

$U_{20}$ & \begin{tabular}{l} $\{\uparrow*,\{1|*\}|0,*,\{0,*|-1\}\},\{0,*|\downarrow\},\{\frac{1}{2},\{1|*\}|0,\downarrow*,\{0,*|-1\}\},\{\uparrow,\uparrow*|0,*\},\{\uparrow*,\{1|*\}|0,\downarrow*\},\{1|\downarrow,\downarrow*\},\{\uparrow*,\{1|*\}|*,\downarrow\},$ \\

$\{1,\{1|0\}|\downarrow,\downarrow*,\{0,*|-1\}\},\{0,\{1|0\}|0,*,*2\},\{1,\{1|0\}|0,\{0|-1\}\},\{*,\uparrow,\{1|0,*\}|0,*\},\{\frac{1}{2},\{1|0\},\{1|*\}|*,\{*|-1\}\},\{0,*|0,\{0,*|-1\}\},$ \\

$\{1|0,*,\{0,*|-1\}\},\{\uparrow,\{1|0\}|0,*,\{0,*|-1\}\}, \uparrow^{2},\{\uparrow,\uparrow*,\{1|0,*\}|0,*,*2\},\{\frac{1}{2},\{1|0,*\}|0,*,\{0,*|-1\}\},\{\frac{1}{2}|0,*,*2\},\{\frac{1}{2},\{1|0,*\}|0,\downarrow*\},$ \\

$\{\frac{1}{2},\{1|0,*\}|*,\downarrow\},\{\frac{1}{2},\{1|*\}|*,\downarrow,\{0,*|-1\}\},\{\{1|0\},\{1|*\}|*,\downarrow,\{0,*|-1\}\},\{0|*,\{0,*|-1\}\},\{\uparrow|\downarrow*,\{0,*|-1\}\},\{1,\{1|*\}|0,\{0|-1\}\},$ \\

$\{\uparrow,\{1|0\}|*,\downarrow\},\{0,\uparrow*,\{1|0,*\}|0,*\},\{\frac{1}{2}||*|-1\},\{\uparrow,\{1|0\}|0,\downarrow*\},\{0,\uparrow*|\downarrow,\{0,*|-1\}\},\{1,\{1|0,*\}|\downarrow,\downarrow*,\pm1\},$ \\

$\{\{1|0\},\{1|*\}|0,\downarrow*,\{0,*|-1\}\},\{\frac{1}{2},\{1|0\},\{1|*\}|0,\{0|-1\}\},\{1,1*|\{0|-1\},\{*|-1\}\},\{2|\{0|-1\},-\frac{1}{2},\{*|-1\}\},$ \\

$\{1,\{1|*\}|\downarrow, \downarrow*,\{0,*|-1\}\},\{\frac{1}{2},\{1|0\}|\downarrow,\downarrow*\},\{1|*||0|-1\},\{1*|\downarrow*,\{*|-1\}\},\{\frac{1}{2},\{1|0\},\{1|*\}|\downarrow,\downarrow*,\{0,*|-1\}\},$ \\

$\{1,\{1|0\},\{1|*\}|\downarrow,\{0|-1\}\},\{0,\{1|0\}||*|-1\},\{1,\{1|*\}|\pm1,-\frac{1}{2}\},\{1|0,\downarrow*,\pm1\},\{\frac{1}{2},\{1|*\}|\downarrow,\downarrow*\},\{\frac{1}{2},\{1|0\}|0,\downarrow*,\{0,*|-1\}\},$ \\

$\{1|*,\downarrow,\pm1\},\{1,\{1|0,*\}|0,\downarrow*,\{0,*|-1\}\},\{1,\{1|*\}|*,\{*|-1\}\},\{\frac{1}{2},\{1|0\},\{1|*\}|-\frac{1}{2}\},\{1,\{1|*\}|-1\},\{\frac{1}{2},\{1|0\}|*,\downarrow,\{0,*|-1\}\},$ \\

$\{\frac{1}{2}||0|-1\},\{1,\{1|0,*\}|-\frac{1}{2}\},\{1,\{1|0\},\{1|*\}|\downarrow*,\{*|-1\}\},\{1,\{1|0\},\{1|*\}|-\frac{1}{2},\{0,*|-1\}\},\{1,\{1|0\},\{1|*\}|-1,\pm1\},\{1*|\downarrow,\{0|-1\}\},$ \\

$\{1,\{1|0\}|-1\},\{1*|-\frac{1}{2},\{0,*|-1\}\},\{1,1*|-\frac{1}{2},\{*|-1\}\},\{1,\{1|0,*\}|*,\downarrow,\{0,*|-1\}\},\{1,1*|-1,\{0,*|-1\}\},\{1,1*|\{0|-1\},-\frac{1}{2}\},$ \\

$\{*,\{1|*\}|0,*,*2\},\{\{1|0\},\{1|*\}|\downarrow, \downarrow*\},\{1,\{1|0\}|\pm1,-\frac{1}{2}\},\{1*|-1\},\{1,\{1|0\}|*,\{*|-1\}\},\{2|-1,\{*|-1\}\},\{2|-1,\{0|-1\}\}$
 \end{tabular} \\ \hline 

$U_{21}$ & \begin{tabular}{l} $\{\uparrow*,\{1|*\}|*,\downarrow,\{0,*|-1\}\},\{*,\{1|*\}|*,\downarrow\},\{\uparrow*,\{1|*\}|0,\downarrow*,\{0,*|-1\}\},\{\uparrow,\uparrow*|0,*,*2\},\{*,\{1|*\}|0,\downarrow*\},\{1|-\frac{1}{2}\},\{\uparrow, \uparrow*,\{1|0,*\}|*,\downarrow\},$ \\

$\{\frac{1}{2},\{1|0\}|\downarrow,\downarrow*,\{0,*|-1\}\},\{0,\{1|0\}|0,*,\{0,*|-1\}\},\{1,\{1|0\}|\downarrow,\{0|-1\}\},\{0,*,\{1|0,*\}|0,*\},\{\frac{1}{2},\{1|0\}|*,\{*|-1\}\},$ \\

$\{*,\{1|*\}|0,*,\{0,*|-1\}\},\{1|0,\downarrow*,\{0,*|-1\}\},\{\uparrow,\{1|0\}|0,\downarrow*,\{0,*|-1\}\},\{0,\{1|0\}|0,\downarrow*\},\{0,\uparrow*,\{1|0,*\}|0,*,*2\},$ \\

$\{\uparrow, \uparrow*,\{1|0,*\}|0,*,\{0,*|-1\}\},\{\frac{1}{2}|0,*,\{0,*|-1\}\},\{\frac{1}{2}|0,\downarrow*\},\{\frac{1}{2}|*,\downarrow\},\{\frac{1}{2},\{1|*\}|*,\{*|-1\}\},\{\uparrow,\{1|0\}|*,\downarrow,\{0,*|-1\}\},$ \\

$\{0|\downarrow*,\{0,*|-1\}\},\{\uparrow||0|-1\},\{\frac{1}{2},\{1|*\}|0,\{0|-1\}\},\{0,\{1|0\}|*,\downarrow\},\{0,\uparrow*|0,*\},\{0,\uparrow*||*|-1\},\{\uparrow,\uparrow*,\{1|0,*\}|0,\downarrow*\},$ \\

$\{0,*|\downarrow,\{0,*|-1\}\},\{\frac{1}{2},\{1|0,*\}|\downarrow,\downarrow*\},\{\{1|0\},\{1|*\}|\downarrow,\downarrow*,\{0,*|-1\}\},\{\frac{1}{2},\{1|0\},\{1|*\}|\downarrow,\{0|-1\}\},\{1,1*|\{0|-1\},-\frac{1}{2},\{*|-1\}\},$ \\

$\{2|-1,\{0|-1\},\{*|-1\}\},\{1,\{1|*\}|\downarrow,\{0|-1\}\},\{\uparrow,\{1|0\}|\downarrow,\downarrow*\},\{\{1|0\},\{1|*\}|0,\{0|-1\}\},\{1*|\{0|-1\},\{*|-1\}\},$ \\

$\{\frac{1}{2},\{1|*\}|\downarrow,\downarrow*,\{0,*|-1\}\},\{1,\{1|0\},\{1|*\}|\{0|-1\},\{*|-1\}\},\{\{1|0\},\{1|*\}|*,\{*|-1\}\},\{\frac{1}{2},\{1|*\}|-\frac{1}{2}\},\{1|\downarrow, \downarrow*,\pm1\},\{\uparrow*,\{1|*\}|\downarrow,\downarrow*\},$ \\

$\{\frac{1}{2},\{1|0,*\}|0,\downarrow*,\{0,*|-1\}\},\{1|*,\downarrow,\{0,*|-1\}\},\{1,\{1|0,*\}|0,\{0|-1\}\},\{1,\{1|*\}|\downarrow*,\{*|-1\}\},\{\frac{1}{2},\{1|0\},\{1|*\}|-1\},$ \\

$\{1,\{1|0,*\}|-1\},\{\frac{1}{2},\{1|0,*\}|*,\downarrow,\{0,*|-1\}\},\{\frac{1}{2},\{1|0\}|0,\{0|-1\}\},\{1,\{1|0,*\}|\pm1,-\frac{1}{2}\},\{\frac{1}{2},\{1|0\},\{1|*\}|\downarrow*,\{*|-1\}\},$ \\ 

$\{1,\{1|*\}|-\frac{1}{2},\{0,*|-1\}\},\{1,\{1|*\}|-1,\pm1\},\{1*|\{0|-1\},-\frac{1}{2}\},\{1,\{1|0\}|-1,\pm1\},\{\frac{1}{2},\{1|0\},\{1|*\}|-\frac{1}{2},\{0,*|-1\}\},$ \\

$\{1,\{1|0\},\{1|*\}|-\frac{1}{2},\{*|-1\}\},\{1,\{1|0,*\}|\downarrow,\downarrow*,\{0,*|-1\}\},\{1,\{1|0\},\{1|*\}|-1,\{0,*|-1\}\},\{1,\{1|0\},\{1|*\}|\{0|-1\},-\frac{1}{2}\},$ \\

$\{*,\uparrow,\{1|0,*\}|0,*,*2\},\{\{1|0\},\{1|*\}|-\frac{1}{2}\},\{\frac{1}{2},\{1|0\}|-\frac{1}{2}\},\{1*|-1,\{0,*|-1\}\},\{1,\{1|0\}|\downarrow*,\{*|-1\}\},\{1,1*|-1,\{*|-1\}\},$ \\

$\{1,1*|-1,\{0|-1\}\},\{1*|-\frac{1}{2},\{*|-1\}\},\{1,\{1|0\}|-\frac{1}{2},\{0,*|-1\}\},\{1,\{1|0,*\}|*,\{*|-1\}\},\{*,\uparrow|0,*\},\{2|-1*\}$  \end{tabular} \\ \hline

$U_{22}$ &

\begin{tabular}{l} $\{\uparrow*,\{1|*\}|\downarrow,\downarrow*,\{0,*|-1\}\},\{*,\{1|*\}|*,\downarrow,\{0,*|-1\}\},\{\uparrow*,\{1|*\}|0,\{0|-1\}\},\{\uparrow, \uparrow*|0,*,\{0,*|-1\}\},\{*,\{1|*\}|\downarrow,\downarrow*\},\{1|\pm1,-\frac{1}{2}\},$ \\

$\{*,\uparrow,\{1|0,*\}|*,\downarrow\},\{\frac{1}{2},\{1|0,*\}|\downarrow,\downarrow*,\{0,*|-1\}\},\{0,\uparrow*,\{1|0,*\}|0,*,\{0,*|-1\}\},\{1,\{1|0\}|\{0|-1\},\{*|-1\}\},\{0,*,\{1|0,*\}|0,*,*2\},$ \\

$\{\frac{1}{2},\{1|0\}|\downarrow*,\{*|-1\}\},\{*,\{1|*\}|0,\downarrow*,\{0,*|-1\}\},\{1|0,\{0|-1\}\},\{0,\{1|0\}|0,\downarrow*,\{0,*|-1\}\},\{0,\{1|0\}|\downarrow,\downarrow*\},\{0,\uparrow*,\{1|0,*\}|0,\downarrow*\},$ \\

$\{*,\uparrow,\{1|0,*\}|0,*,\{0,*|-1\}\},\{\frac{1}{2}|0,\downarrow*,\{0,*|-1\}\},\{\frac{1}{2}|\downarrow,\downarrow*\},\{\uparrow,\uparrow*|*,\downarrow\},\{\frac{1}{2},\{1|0,*\}|*,\{*|-1\}\},\{0,\{1|0\}|*,\downarrow,\{0,*|-1\}\},\text{Tiny}_1,$ \\

$\{\uparrow,\{1|0\}|0,\{0|-1\}\},\{\frac{1}{2},\{1|*\}|\downarrow,\{0|-1\}\},\{0,\uparrow*,\{1|0,*\}|*,\downarrow\},\{0,\uparrow*|0,*,*2\},\{\uparrow*,\{1|*\}|*,\{*|-1\}\},\{\uparrow,\uparrow*|0,\downarrow*\},\{0,*||*|-1\},$ \\

$\{\frac{1}{2},\{1|0,*\}|-\frac{1}{2}\},\{\{1|0\},\{1|*\}|\downarrow*,\{*|-1\}\},\{\frac{1}{2},\{1|0\},\{1|*\}|\{0|-1\},-\frac{1}{2}\},\{1,1*|-1,\{0|-1\},\{*|-1\}\},\{2|-1,-1*\},$ \\

$\{1,\{1|*\}|\{0|-1\},-\frac{1}{2}\},\{\uparrow,\{1|0\}|\downarrow,\downarrow*,\{0,*|-1\}\},\{\{1|0\},\{1|*\}|\downarrow,\{0|-1\}\},\{1*|\{0|-1\},-\frac{1}{2},\{*|-1\}\},\{\frac{1}{2},\{1|*\}|\downarrow*,\{*|-1\}\},$ \\

$\{1,\{1|*\}|\{0|-1\},\{*|-1\}\},\{\uparrow,\{1|0\}|*,\{*|-1\}\},\{\frac{1}{2},\{1|*\}|-\frac{1}{2},\{0,*|-1\}\},\{1|\downarrow,\downarrow*,\{0,*|-1\}\},\{\uparrow,\uparrow*,\{1|0,*\}|\downarrow,\downarrow*\},$ \\

$\{\uparrow,\uparrow*,\{1|0,*\}|0,\downarrow*,\{0,*|-1\}\},\{\frac{1}{2}|*,\downarrow,\{0,*|-1\}\},\{\frac{1}{2},\{1|0,*\}|0,\{0|-1\}\},\{1,\{1|*\}|-\frac{1}{2},\{*|-1\}\},\{\frac{1}{2},\{1|*\}|-1\},\{1,\pm1|-1\},$ \\

$\{\uparrow,\uparrow*,\{1|0,*\}|*,\downarrow,\{0,*|-1\}\},\{\frac{1}{2},\{1|0\}|\downarrow,\{0|-1\}\},\{1,\{1|0,*\}|-\frac{1}{2},\{0,*|-1\}\},\{\frac{1}{2},\{1|0\},\{1|*\}|\{0|-1\},\{*|-1\}\},$ \\ 

$\{1,\{1|*\}|-1,\{0,*|-1\}\},\{1,\{1|0,*\}|-1,\pm1\},\{1*|-1,\{0|-1\}\},\{\frac{1}{2},\{1|0\}|-1\},\{\{1|0\},\{1|*\}|-\frac{1}{2},\{0,*|-1\}\},$ \\

$\{\frac{1}{2},\{1|0\},\{1|*\}|-\frac{1}{2},\{*|-1\}\},\{1,\{1|0,*\}|\downarrow,\{0|-1\}\},\{1,\{1|0\}|-1,\{0,*|-1\}\},\{1,\{1|0\},\{1|*\}|\{0|-1\},-\frac{1}{2},\{*|-1\}\},$ \\

$\{*,\uparrow,\{1|0,*\}|0,\downarrow*\}, \{\uparrow*,\{1|*\}|-\frac{1}{2}\},\{\uparrow,\{1|0\}|-\frac{1}{2}\},\{\frac{1}{2},\{1|0\},\{1|*\}|-1,\{0,*|-1\}\},\{1,\{1|0,*\}|\downarrow*,\{*|-1\}\},$ \\

$\{1,\{1|0\},\{1|*\}|-1,\{*|-1\}\},\{1,\{1|0\},\{1|*\}|-1,\{0|-1\}\},\{1*|-1,\{*|-1\}\},\{\frac{1}{2},\{1|0\}|-\frac{1}{2},\{0,*|-1\}\},\{1|*,\{*|-1\}\},$ \\

$\{0,*,*2|0,*\},\{1,1*|-1*\},\{\{1|0\},\{1|*\}|-1\},\{1,\{1|0\}|\{0|-1\},-\frac{1}{2}\},\{*,\uparrow|0,*,*2\},\{1,\{1|0\}|-\frac{1}{2},\{*|-1\}\}$

\end{tabular}
 \\ \hline

$U_{23}$ &
\begin{tabular}{l}
 $\{\uparrow*,\{1|*\}|-\frac{1}{2},\{0,*|-1\}\},\{*,\{1|*\}|\downarrow,\downarrow*,\{0,*|-1\}\},\{\uparrow*,\{1|*\}|\downarrow,\{0|-1\}\},\{0,\uparrow*|0,*,\{0,*|-1\}\},\{*,\uparrow,\{1|0,*\}|\downarrow,\downarrow*\},\pm\frac{1}{2},$ \\
 
  $\{0,*,\{1|0,*\}|*,\downarrow\},\{\frac{1}{2}|\downarrow,\downarrow*,\{0,*|-1\}\},\pm(0,\uparrow*,\{1|0,*\}),\{\frac{1}{2},\{1|0\}|\{0|-1\},\{*|-1\}\},\pm(0,*,\{1|0,*\}),\{\uparrow,\{1|0\}|\downarrow*,\{*|-1\}\},$ \\
 $\{*,\{1|*\}|0,\{0|-1\}\},\{1|\downarrow,\{0|-1\}\},\pm(0,\{1|0\}),\{0,\{1|0\}|-\frac{1}{2}\},\pm(0,\uparrow*),\pm(*,\uparrow,\{1|0,*\}),\{\uparrow,\uparrow*|0,\downarrow*,\{0,*|-1\}\},\pm(\uparrow,\uparrow*),\pm(*,\uparrow),$
 \\
$\{\frac{1}{2}|*,\{*|-1\}\},\{0,\{1|0\}|\downarrow,\downarrow*,\{0,*|-1\}\},0,\{\uparrow,\uparrow*,\{1|0,*\}|0,\{0|-1\}\},\{\frac{1}{2},\{1|*\}|\{0|-1\},-\frac{1}{2}\},\{0,\uparrow*,\{1|0,*\}|*,\downarrow,\{0,*|-1\}\},$ \\

$\{0,\uparrow*|*,\downarrow\},\{\uparrow,\uparrow*,\{1|0,*\}|*,\{*|-1\}\},\{*,\uparrow|0,\downarrow*\},\pm(*,\{1|*\}),\{\uparrow,\uparrow*,\{1|0,*\}|-\frac{1}{2}\},\pm(\uparrow*,\{1|*\}),\pm(\frac{1}{2},\{1|0\}),\pm(1,\{1|0\},\{1|*\}),$ \\

$\pm(1,1*),\{1,\{1|*\}|-1,\{0|-1\}\},\pm(\uparrow,\{1|0\}),\pm(\{1|0\},\{1|*\}),\pm(\frac{1}{2},\{1|0\},\{1|*\}),\{\frac{1}{2},\{1|0,*\}|\downarrow*,\{*|-1\}\},$ \\

$\{1,\{1|*\}|\{0|-1\},-\frac{1}{2},\{*|-1\}\},\{0,\{1|0\}|*,\{*|-1\}\},\{\frac{1}{2},\{1|*\}|-1,\{0,*|-1\}\},\{1|-\frac{1}{2},\{0,*|-1\}\},\{0,\uparrow*,\{1|0,*\}|\downarrow,\downarrow*\},$ \\ 

$\{*,\uparrow,\{1|0,*\}|0,\downarrow*,\{0,*|-1\}\},\{\uparrow, \uparrow*|*,\downarrow,\{0,*|-1\}\},\{\frac{1}{2}|0,\{0|-1\}\},\pm(\frac{1}{2},\{1|*\}),\{\uparrow*,\{1|*\}|-1\},\pm(1,\pm1),\pm(\uparrow,\uparrow*,\{1|0,*\}),$ \\ 

$\{\frac{1}{2},\{1|0,*\}|\downarrow,\{0|-1\}\},\pm(\frac{1}{2},\{1|0,*\}),\{\frac{1}{2},\{1|*\}|\{0|-1\},\{*|-1\}\},\pm(1,\{1|*\}),\{\frac{1}{2},\{1|0,*\}|-1\},\{\frac{1}{2},\{1|0\},\{1|*\}|-1,\{0|-1\}\},$ \\ 

$\{\uparrow,\{1|0\}|-1\},\{\{1|0\},\{1|*\}|\{0|-1\},-\frac{1}{2}\},\{\{1|0\},\{1|*\}|-\frac{1}{2},\{*|-1\}\},\{1,\{1|0,*\}|\{0|-1\},\{*|-1\}\},\pm(1,\{1|0,*\}),$ \\

$\{1,\{1|0\}|\{0|-1\},-\frac{1}{2},\{*|-1\}\},\{0,*,\{1|0,*\}|0,\downarrow*\},\{*,\{1|*\}|-\frac{1}{2}\},\{\uparrow,\{1|0\}|-\frac{1}{2},\{0,*|-1\}\},\{\frac{1}{2},\{1|0\},\{1|*\}|-1,\{*|-1\}\},$ \\ 

$\{1,\{1|0,*\}|-\frac{1}{2},\{*|-1\}\},\{1,\{1|0\}|-1,\{*|-1\}\}, \pm(1,\{1|0\}),\{1*|-1,\{0|-1\},\{*|-1\}\},\{\frac{1}{2},\{1|0\}|-1,\{0,*|-1\}\},$ \\

$\{1|\downarrow*,\{*|-1\}\},*2,\{1,\{1|0\},\{1|*\}|-1*\},\{\{1|0\},\{1|*\}|-1,\{0,*|-1\}\}, \{1,\{1|0,*\}|\{0|-1\},-\frac{1}{2}\},\{*,\uparrow|0,*,\{0,*|-1\}\},$ \\

$\{\frac{1}{2},\{1|0\}|-\frac{1}{2},\{*|-1\}\},\pm1,*,*3,\pm(1*),\pm2$ 
 \end{tabular}
 \\ \hline
 
\end{tabular}
\end{table*}

\end{document}